\documentclass[11pt,a4paper]{article}

\usepackage[T1]{fontenc}
\usepackage{lmodern}
\usepackage{microtype}
\usepackage{amsmath,amssymb,amsthm,mathtools}
\usepackage[a4paper,left=1.08in,right=1.08in,top=0.96in,bottom=0.98in]{geometry}
\usepackage{enumitem}
\usepackage[colorlinks=true,linkcolor=blue,citecolor=blue,urlcolor=blue]{hyperref}

\usepackage{lineno}
\setlist{nosep,leftmargin=2.0em}

\hypersetup{
	colorlinks=true,
	linkcolor=blue,
	filecolor=blue,
	urlcolor=blue,
	citecolor=cyan,
}

\newtheorem{theorem}{Theorem}[section]
\newtheorem{lemma}[theorem]{Lemma}
\newtheorem{proposition}[theorem]{Proposition}
\newtheorem{corollary}[theorem]{Corollary}
\newtheorem{question}[theorem]{Question}

\newcommand{\inj}{\mathrm{inj}}
\newcommand{\ex}{\mathrm{ex}}
\newcommand{\N}{\mathcal N}
\newcommand{\ed}{\mathrm{dist}_{\mathrm{edit}}}
\newcommand{\fall}[2]{(#1)_{#2}}

\title{Every graph is eventually Tur\'an-good and Tur\'an-stable}

\author{Jiabao Yang\footnote{Email: jbyang1215@nju.edu.cn}\\
\small{School of Mathematics, Nanjing University, Nanjing 210093, P.R. China}}

\date{}

\begin{document}
\maketitle

\begin{abstract}
Let $T_r(n)$ denote the complete $r$-partite graph on $n$ vertices whose part sizes differ by at most one. 
A graph $H$ is called $K_{r+1}$-Tur\'an-good if, for all sufficiently large $n$, the graph $T_r(n)$ contains the maximum number of copies of $H$ among all $n$-vertex $K_{r+1}$-free graphs. 
We say that $H$ is $K_{r+1}$-Tur\'an-stable if, for every $\varepsilon>0$, there exist $\delta>0$ and $n_0$ such that, whenever $n\ge n_0$ and an $n$-vertex $K_{r+1}$-free graph $G$ contains at least $\mathrm{ex}(n,H,K_{r+1})-\delta n^{v(H)}$ copies of $H$, the edit distance between $G$ and $T_r(n)$ is at most $\varepsilon n^2$,
where $v(H)$ denotes the order of $H$.

In this paper, we prove that every graph $H$ is both $K_{r+1}$-Tur\'an-good and $K_{r+1}$-Tur\'an-stable for $r\ge 4v(H)^3+11v(H)^2$. 
This strengthens results of Morrison, Nir, Norin, Rz\k{a}\.{z}ewski, and Wesolek~[J. Combin. Theory Ser. B, 2023] and Gerbner and Hama Karim~[J. Graph Theory, 2024], and gives a positive answer to a question of Morrison, Nir, Norin, Rz\k{a}\.{z}ewski, and Wesolek. We also prove that
$\inj(H,G)\le \inj(H,T_r(n))$
for every $n$-vertex $K_{r+1}$-free graph $G$ and for $r\ge 40v(H)^3$, where $\inj(H,G)$ denotes the number of injective homomorphisms from $H$ to $G$. Finally, we give a negative answer to another question of Morrison, Nir, Norin, Rz\k{a}\.{z}ewski, and Wesolek.
\end{abstract}


\section{Introduction}
For a graph $H$, we write $v(H)=|V(H)|$ and $e(H)=|E(H)|$.
We say a graph $G$ is {\it $F$-free} if $G$ does not contain a copy of $F$ as a subgraph.
A basic result in extremal Combinatorics is Turán's theorem. 
Let $K_{r+1}$ be the complete graph on $r+1$ vertices.
In 1941, Tur\'{a}n \cite{turan1941} proved that the unique $n$-vertex $K_{r+1}$-free graph with maximum number of edges is  the Tur\'an graph $T_r(n)$,
which is the complete $r$-partite graph where each partite class has cardinality $\lfloor n/r\rfloor$ or $\lceil n/r \rceil$. 

In 2016, Alon and Shikhelman \cite{Alon2016} initiated the systematic study of the generalized Tur\'{a}n problem. 
For a pair of graph $H$ and $G$, let $\N(H,G)$ be the number of subgraphs of $G$ isomorphic to $H$.
The generalized Tur\'{a}n number which is the maximum possible number of copies of $H$ in $F$-free graphs of order $n$ is denoted by $\mathrm{ex}(n,H,F)$, that is 
$$\mbox{ex}(n,H,F):=\max\{\N(H,G):v(G)=n \mbox{ and }G \mbox{ is $F$-free}\}.$$
Since then, the generalized Tur\'{a}n problems has received a lot of attentions, see e.g.
\cite{Chenya2024,Erdos1962,Furedi2015,Gerbner2020,Gerbner2019,Luo2017,Lv2024,Yang2024,ZhangChenGyoriZhu2024}.

However, there are not many exact results in this area.
In many cases where the extremal problem is concrete, the extremal graph turns out to be the Tur\'an graph. 
The chromatic number $\chi(H)$ of a graph $H$ is the minimum number of colors needed to color the vertices of $H$ such that no two adjacent vertices share the same color.
Let $F$ be a graph with chromatic number $r+1$. 
A graph $H$ is called $F$\emph{-Tur\'an-good} if, for all sufficiently large $n$, 
the Tur\'an graph $T_r(n)$ contains the maximum possible number of copies of $H$ among all $n$-vertex $ F$-free graphs. 
In other words, $\mathrm{ex}(n, H, F) = N(H, T_r(n))$.
Although the term Tur\'an-good was introduced only recently by Gerbner and Palmer~\cite{Gerbner2022}, problems of this type were studied much earlier by Gy\H{o}ri, Pach, and Simonovits~\cite{Gyori1991}.

Gerbner and Palmer~\cite{Gerbner2022} conjectured that, for every graph $H$, there exists an integer $r_0=r_0(H)$ such that $H$ is $K_{r+1}$-Tur\'an-good when $r\ge r_0(H)$. 
This conjecture has been proved for several families of graphs, including stars~\cite{Cutler2022}, complete multipartite graphs~\cite{Gerbner2022}, paths~\cite{Gerbner2023}, and the cycle $C_5$~\cite{Lidicky2021}. 
Recently, Morrison, Nir, Norin, Rz\k{a}\.{z}ewski, and Wesolek~\cite{Morrison2023} proved the conjecture holds with $r_0=300v(H)^9$.

\begin{theorem}[Morrison, Nir, Norin, Rz\k{a}\.{z}ewski, and Wesolek~\cite{Morrison2023}]\label{thm:Morrison2023}
Let $H$ be a graph and $r\ge 300v(H)^9$. Then $H$ is $K_{r+1}$-Tur\'an-good.
\end{theorem}

Let $\inj(H,G)$ be the number of injective maps $\varphi:V(H)\to V(G)$ such that $\varphi(x)\varphi(y)\in E(G)$ when $xy\in E(H)$. 
We call each injective map a {\em homomorphism} from $H$ to $G$.
Every copy of $H$ gives exactly $|\mathrm{Aut}(H)|$ such maps, where $|\mathrm{Aut}(H)|$ is the order of automorphism group of $H$,
and hence $\inj(H,G)=|\mathrm{Aut}(H)|\N(H,G)$. 
It follows the fact that a graph $H$ is $K_{r+1}$-Tur\'an-good if and only if for every $n$-vertex $K_{r+1}$-free graph $G$, 
we have $\mathrm{inj}(H, G) \le \operatorname{inj}(H, T_r(n))$.

Morrison, Nir, Norin, Rzążewski, and Wesolek proved the following stronger theorem, which implies Theorem \ref{thm:Morrison2023}.

\begin{theorem}[Morrison, Nir, Norin, Rz\k{a}\.{z}ewski, and Wesolek~\cite{Morrison2023}]\label{thm:Morrison2023-stronger}
Let $H$ be a graph and $r\ge 300v(H)^9$. 
Then $\inj(H,G)\le\inj(H,T_r(n))$ for every $n$-vertex $K_{r+1}$-free graph $G$
\end{theorem}

The {\em edit distance} of two $n$-vertex graphs $G$ and $G'$, denoted by $\ed(G,G')$, is the minimum number of edge additions and deletions needed to make them isomorphic. 
%We say that a graph $H$ is {\em $K_{r+1}$-Tur\'an-stable} if every $K_{r+1}$-free graph $G$ on $n$ vertices with 
%$\mathrm{ex}(n,H,K_{r+1})- o(n^{v(H)})$ copies of $H$ has edit distance $o(n^2)$ from the Tur\'an graph.
%We say that $H$ is {\em $K_{r+1}$-Tur\'an-stable} if, for every $\varepsilon>0$, there are $\delta>0$ and $n_0$ such that 
%every $n\ge n_0$ and every $n$-vertex $K_{r+1}$-free graph $G$ satisfying $\N(H,G)\ge\N(H,T_r(n))-\delta n^{v(H)}$ also satisfies $\ed(G,T_r(n))\le\varepsilon n^2$.
We say that $H$ is $K_{r+1}$-Tur\'an-stable if, for every $\varepsilon>0$, 
there exist $\delta>0$ and $n_0$ such that, whenever $n\ge n_0$ and an $n$-vertex $K_{r+1}$-free 
graph $G$ contains at least $\mathrm{ex}(n,H,K_{r+1})-\delta n^{v(H)}$ copies of $H$, 
$\ed(G,T_r(n))\le\varepsilon n^2$.

Using these notions, the classical Erd\H{o}s--Simonovits stability theorem~\cite{Erdos1966,Erdos1968,Simonovits1974} implies that $K_2$ is $K_{r+1}$-Tur\'an-stable for every $r\ge 2$. 
The first stability result for generalized Tur\'an problems was obtained by Ma and Qiu~\cite{MaQiu2018}, 
who proved that $K_k$ is $K_{r+1}$-Tur\'an-stable for every $r\ge k$. 
The general notion of a Tur\'an-stable graph was introduced by Gerbner~\cite{Gerbner}.

Morrison, Nir, Norin, Rz\k{a}\.{z}ewski, and Wesolek asked the following question.

\begin{question}[Morrison, Nir, Norin, Rz\k{a}\.{z}ewski, and Wesolek~\cite{Morrison2023}]\label{q:stability}
Fix a graph $H$ and let $r$ be large enough that $H$ is $K_{r+1}$-Tur\'an-good. Does it follow that $H$ is $K_{r+1}$-Tur\'an-stable?
\end{question}

Gerbner and Hama Karim~\cite{GerbnerKarim} proved Tur\'an-stability under the same condition. 
Thus they gave an affirmative answer to Question~\ref{q:stability} in this range.

\begin{theorem}[Gerbner and Hama Karim~\cite{GerbnerKarim}]\label{thm:GerbnerKarim}
Let $H$ be a graph and $r\ge 300v(H)^9$. Then $H$ is $K_{r+1}$-Tur\'an-stable.
\end{theorem}

Chen and Liu \cite{Chenliu2024} later developed a stronger degree-stability and extended it to a hypergraph framework.
In \cite{Lidicky2021} it is conjectured that Theorem \ref{thm:Morrison2023} should hold with $r_0(H) = v(H) +1$. 
Motivated by the above results, we first prove the following theorems.

\begin{theorem}\label{thm:cubic}
Let $H$ be a graph and $r\ge4v(H)^3+11v(H)^2$. Then $H$ is $K_{r+1}$-Tur\'an-good and $K_{r+1}$-Tur\'an-stable.
\end{theorem}

\begin{theorem}\label{thm:alln}
Let $H$ be a graph and $r\ge40v(H)^3$. Then $\inj(H,G)\le\inj(H,T_r(n))$ for every positive integer $n$ and every $n$-vertex $K_{r+1}$-free graph $G$.
\end{theorem}

Thus Theorem~\ref{thm:cubic} not only strengthens Theorems~\ref{thm:Morrison2023} and~\ref{thm:GerbnerKarim}, 
but also provides a strengthened affirmative answer to Question~\ref{q:stability} in the general regime 
$r \ge 4v(H)^3 + 11v(H)^2$. 
Moreover, Theorem~\ref{thm:alln} establishes an exact inequality for every order $n$ under the condition 
$r \ge 40v(H)^3$, and in this range it strengthens Theorem~\ref{thm:Morrison2023-stronger}.

Apart from trying to lower the general threshold above, another interesting problem is to study whether Tur\'an-goodness is monotone in the forbidden clique. 
Morrison, Nir, Norin, Rz\k{a}\.{z}ewski, and Wesolek also asked the following question.

\begin{question}[Morrison, Nir, Norin, Rz\k{a}\.{z}ewski, and Wesolek~\cite{Morrison2023}]\label{q:monotonicity}
If $H$ is $K_r$-Tur\'an-good, must $H$ be $K_{r+1}$-Tur\'an-good?
\end{question}

Fix $r\ge3$ and let $A_1,\ldots,A_{r-1}$ be disjoint sets. For $s\ge4$ and $0\le\ell\le s$, let $H_{r,s,\ell}$ be obtained from the complete $(r-1)$-partite graph with $|A_i|=s$ by deleting a matching of size $\ell$ between $A_1$ and $A_2$. Our next theorem gives a negative answer to Question~\ref{q:monotonicity} for every $r\ge3$.

\begin{theorem}\label{thm:counterexamples}
Let $r\ge3$ and $s\ge12\lceil\log_2 r\rceil$. Then every $H_{r,s,\ell}$ with $0\le\ell\le s$ is $K_r$-Tur\'an-good but is not $K_{r+1}$-Tur\'an-good.
\end{theorem}

A natural question is whether this is the only family of graphs that serves as a counterexample.
It is worth mentioning that two further families, denoted by $G_{r,s,t}$ and $J_{r,s,d}$, are also counterexamples. 
To keep the main result concise, we give their definitions and complete proofs in Appendix~\ref{app:further-counterexamples}.

The proof of Theorem~\ref{thm:counterexamples} uses a simple comparison at the next clique level. 
The graph $T_{r-1}(n)$ can be viewed as an $r$-partite graph with one empty part. 
We determine the proper $(r-1)$-colorings and proper $r$-colorings of $H_{r,s,\ell}$ exactly and show that the first normalized coloring count is larger.

A graph is claw-free if it has no induced copy of $K_{1,3}$. 
Finally, we provide a sufficient condition ensuring that a graph is $K_{k+1}$-Tur\'an-good as follows.

\begin{theorem}\label{thm:claw-stable}
Let $H$ be claw-free with $\chi(H)\le k$. 
Suppose that, for all sufficiently large $n$, some $n$-vertex $K_{k+1}$-free graph attaining $\ex(n,H,K_{k+1})$ can be transformed into a complete $k$-partite graph by adding and deleting $o(n^2)$ edges. 
Then $H$ is $K_{k+1}$-Tur\'an-good.
\end{theorem}

\medskip
{\noindent\bf Notation:}
%Let $\Delta(G,G') := |E(G) \triangle E(G')|$ denote the edit distance between two graphs $G$ and $G'$ on the same vertex set.
For two sets $A$ and $B$, their symmetric difference is $A\triangle B = (A\setminus B)\cup (B\setminus A)$. 
For a graph $G$ and $V\subseteq V(G)$, we write $G[V]$ for the subgraph of $G$ induced by $V$.
For a vertex $v\in V(G)$, let $N_G(v)=\{u\in V(G):uv\in E(G)\}$ denote the neighborhood of $v$ in $G$ 
and let $d_G(v)=|N_G(v)|$ denote the degree of $v$ in $G$. 
We use $G-V$ to denote the graph $G[V(G)\setminus V]$, and denote by $G-v$ the graph of $G\setminus \{v\}$.
For a real number $x$ and a nonnegative integer $a$, write $(x)_a=x(x-1)\cdots(x-a+1)$, with $(x)_0=1$.

The organization of this paper is as follows. 
We first introduce some preliminaries in the next section. 
Section~\ref{sec:cubic} first proves Theorem~\ref{thm:cubic} and then proves Theorem~\ref{thm:alln}. 
In Sections \ref{sec:counterexamples} and \ref{sec:claw}, we prove Theorems~\ref{thm:counterexamples} and \ref{thm:claw-stable}, respectively. 
Appendix~\ref{app:further-counterexamples} proves that $G_{r,s,t}$ and $J_{r,s,d}$ are also counterexamples.

\section{Preliminaries}

A graph $G$ is called $\beta$-dense if $d_G(v)\ge (1-\beta)v(G)$ for every $v\in V(G)$.
We use the following extension lemma of Morrison, Nir, Norin, Rzążewski, and Wesolek~\cite{Morrison2023} directly.

\begin{lemma}[Morrison, Nir, Norin, Rzążewski, and Wesolek~\cite{Morrison2023}]\label{cub:lem:extension}
Let $H$ be a graph and let $X\subseteq V(H)$.
If $G$ is a $\beta$-dense graph, then every injective homomorphism from $H[X]$ to $G$ has 
at least $(1-\beta v(H)(v(H)-|X|))v(G)^{v(H)-|X|}$ extensions to an injective homomorphism from $H$ to $G$.
In particular, $\inj(H,G)\ge(1-\beta v(H)^2)v(G)^{v(H)}$.
\end{lemma}

%Taking $X=\varnothing$ gives
%\begin{equation}\label{cub:eq:global-extension}
%\inj(H,G)\ge(1-\beta v(H)^2)v(G)^q.
%\end{equation}

We will use a stability theorem of F\"uredi~\cite{Furedi}.

\begin{theorem}[F\"uredi~\cite{Furedi}]\label{cub:thm:furedi}
Every $n$-vertex $K_{r+1}$-free graph $G$ contains a spanning $r$-partite subgraph $Q$ such that $e(G)-e(Q)\le e(T_r(n))-e(G)$.
\end{theorem}

The next lemma tells us that we only need to consider graphs with no isolated vertices.

\begin{lemma}\label{cub:lem:isolates}
Let $H_0$ be obtained from $H$ by deleting its $s$ isolated vertices. 
If $v(G)<v(H)$, then $\inj(H,G)=0$. If $v(G)\ge v(H)$, then
$$\inj(H,G)=(v(G)-v(H_0))_s\inj(H_0,G).$$
Consequently, for a fixed forbidden clique, we only need to consider graphs with no isolated vertices for
Tur\'an-goodness, Tur\'an-stability, and the inequality in Theorem~\ref{thm:alln}.
\end{lemma}

\begin{proof}
If $v(G)<v(H)$, there is no injective map from $V(H)$ to $V(G)$, so $\inj(H,G)=0$. 
Suppose that $v(G)\ge v(H)$. After an injective homomorphism from $H_0$ to $G$ has been chosen, 
the $s$ isolated vertices can be mapped injectively to any ordered $s$-tuple of vertices outside its image. 
The number of such ordered tuples is $(v(G)-v(H_0))_s$, which proves the displayed identity.

The Tur\'an-goodness statement and the inequality follow by multiplying both sides of the corresponding inequality for $H_0$ by the same factor. 
For stability, $H_0$ is $K_{r+1}$-Tur\'an-stable, and fix $\varepsilon>0$. 
Let
$\delta_0>0$ and $n_0$ be supplied by the stability of $H_0$. 
Let $\delta=\delta_0/2^s s!$ and $n_1=\max\{n_0,2v(H)\}$.
For every $n\ge n_1$ and every $0\le i\le s-1$, we have
$n-v(H_0)-i\ge n/2$. Hence
$$\binom{n-v(H_0)}{s}
=\frac{(n-v(H_0))_s}{s!}
\ge\frac{n^s}{2^s s!}.$$

Let $G$ be an $n$-vertex $K_{r+1}$-free graph with $n\ge n_1$ and suppose that $\N(H,G)\ge\ex(n,H,K_{r+1})-\delta n^{v(H)}$.
Using the two scaling identities above and then dividing by $\binom{n-v(H_0)}{s}$, we obtain
$$\begin{aligned}
\N(H_0,G)
&\ge \ex(n,H_0,K_{r+1})
-\frac{\delta n^{v(H)}}{\binom{n-v(H_0)}{s}}\\
&\ge \ex(n,H_0,K_{r+1})
-2^s s!\delta n^{v(H)-s}\\
&=\ex(n,H_0,K_{r+1})-\delta_0n^{v(H_0)}.
\end{aligned}$$
The $K_{r+1}$-Tur\'an-stability of $H_0$ implies that $\ed(G,T_r(n))\le\varepsilon n^2$. 
Therefore $H$ is $K_{r+1}$-Tur\'an-stable.
\end{proof}

By Lemma~\ref{cub:lem:isolates}, it is enough to prove Theorems~\ref{thm:cubic} and~\ref{thm:alln} when $H$ has no isolated vertices. 
Indeed, deleting isolated vertices can only decrease $v(H)$, so each hypothesis stated in terms of $v(H)$ remains valid for the graph that remains. 
We therefore assume from now on that $H$ has no isolated vertices. 
An oriented edge of $H$ means an ordered pair $(u,v)$ with $uv\in E(H)$;  thus $H$ has $2e(H)$ oriented edges.

For an $n$-vertex graph $G$ with $n\ge v(H)$ and $v\in V(G)$, let $I_G(v)$ be the number of injective homomorphisms from $H$ to $G$ whose image contains $v$. 
Define $\Gamma_n:=(n-1)_{v(H)-1}$ and $x_v:=(n-1-d_G(v))/(n-1)$.

\begin{lemma}\label{cub:lem:vertex-upper}
For every $v\in V(G)$,
\begin{equation*}
I_G(v)\le \Gamma_n\sum_{u\in V(H)}(1-x_v)^{d_H(u)}.
\end{equation*}
Moreover, the function $\Psi(x):=\sum_{u\in V(H)}(1-(1-x)^{d_H(u)})$ is non-decreasing on $[0,1]$ 
and satisfies $\Psi(x)\ge e(H)(2x-(v(H)-2)x^2)$ for every $x\in[0,1]$.
\end{lemma}

\begin{proof}
Fix $u\in V(H)$ and put $d=d_H(u)$. 
An injective homomorphism with $u$ mapped to $v$ must map the $d$ neighbors of $u$ injectively into $N_G(v)$. 
After those images have been selected, we disregard all remaining adjacency constraints and map the other $v(H)-1-d$ vertices injectively to unused vertices. 
Hence the number of such homomorphisms is at most $(d_G(v))_d(n-1-d)_{v(H)-1-d}$. 
If $d_G(v)<d$, this contribution is zero and the required estimate is immediate. 
Suppose that $d_G(v)\ge d$. 
Since $\Gamma_n=(n-1)_d(n-1-d)_{v(H)-1-d}$ and $(d_G(v)-i)/(n-1-i)\le d_G(v)/(n-1)$ for $0\le i<d$,
$$\frac{(d_G(v))_d}{(n-1)_d}=\prod_{i=0}^{d-1}\frac{d_G(v)-i}{n-1-i}\le\left(\frac{d_G(v)}{n-1}\right)^d=(1-x_v)^d.$$
Thus the contribution associated with $u$ is at most $\Gamma_n(1-x_v)^d$. 
Summing over the unique preimage of $v$ proves the lemma.

Since each summand of $\Psi(\cdot)$ is non-decreasing, it follows that $\Psi(\cdot)$ is non-decreasing on $[0,1]$.
For a positive integer $d\le v(H)-1$, we have $1-(1-x)^d=x\sum_{i=0}^{d-1}(1-x)^i$. 
Bernoulli's inequality gives $(1-x)^i\ge1-ix$ for $0\le x\le1$, and therefore $1-(1-x)^d\ge dx-\binom d2x^2\ge dx-d(v(H)-2)x^2/2$. 
Summing this inequality over $u\in V(H)$ and using $\sum_ud_H(u)=2e(H)$ yields $\Psi(x)\ge2e(H)x-e(H)(v(H)-2)x^2$.
\end{proof}

We next give a lower bound for the number of injective homomorphisms into the Tur\'an graph.

\begin{lemma}\label{cub:lem:turan-lower}
Let $n\ge v(H)$. Then
\begin{equation*}
\inj(H,T_r(n))\ge n\Gamma_n-\frac{e(H)n^2}{r(n-1)}\Gamma_n.
\end{equation*}
Consequently, if an $n$-vertex graph $G$ satisfies $\inj(H,G)\ge\inj(H,T_r(n))-\delta n^{v(H)}$, then some vertex $v_0\in V(G)$ satisfies
\begin{equation*}
I_G(v_0)\ge v(H)\Gamma_n-\frac{v(H)e(H)n}{r(n-1)}\Gamma_n-v(H)\delta n^{v(H)-1}.
\end{equation*}
\end{lemma}

\begin{proof}
Let the part sizes of $T_r(n)$ be $a_1,\ldots,a_r$. 
For a fixed edge of $H$, the number of injective maps that send its ends to the same part is $\sum_i a_i(a_i-1)(n-2)_{v(H)-2}$. 
Since $a_i-1\le n/r$, we have $\sum_i a_i(a_i-1)\le n^2/r$. 
Every injective map that is not a homomorphism maps the ends of at least one edge of $H$ into the same part. 
A union bound over the $e(H)$ edges gives
$$\inj(H,T_r(n))\ge(n)_{v(H)}-\frac{e(H)n^2}{r}(n-2)_{v(H)-2}=n\Gamma_n-\frac{e(H)n^2}{r(n-1)}\Gamma_n,$$
as desired. 
Finally, since $\sum_{v\in V(G)}I_G(v)=v(H)\inj(H,G)$ and $\inj(H,G)\ge\inj(H,T_r(n))-\delta n^{v(H)}$, 
by averaging, we have
$$I_G(v_0)\ge v(H)\Gamma_n-\frac{v(H)e(H)n}{r(n-1)}\Gamma_n-v(H)\delta n^{v(H)-1},$$
as required.
\end{proof}

The first observation replaces the estimate for the largest part used in \cite{Morrison2023,GerbnerKarim}.

\begin{lemma}\label{cub:lem:max-cut}
Let $Q$ be a spanning $r$-partite subgraph of a graph $G$ with the maximum possible number of edges. 
Then $d_Q(v)\ge(1-1/r)d_G(v)$ for every $v\in V(G)$.
\end{lemma}

\begin{proof}
Let $V_1,\ldots,V_r$ be a partition of $Q$, and let $v\in V_i$. 
For convenience, we write $d_G(v,V_i)=|N_G(v)\cap V_i|$.
Moving $v$ from $V_i$ to $V_j$ changes the number of crossing edges by $d_G(v,V_i)-d_G(v,V_j)$. 
Maximality gives $d_G(v,V_i)\le d_G(v,V_j)$ for every $j$. Hence $rd_G(v,V_i)\le\sum_jd_G(v,V_j)=d_G(v)$, and therefore $d_Q(v)=d_G(v)-d_G(v,V_i)\ge(1-1/r)d_G(v)$.
\end{proof}

We now sharpen the multipartite rebalancing lemma.
The proof follows the structure of \cite[Lemma 2.3]{Morrison2023}. 

\begin{lemma}\label{cub:lem:transfer}
Let $C=4v(H)^2+2v(H)-4$. 
Suppose $0<\beta\le1/4$ and $F$ is a $\beta$-dense graph. 
Let $A,B\subseteq V(F)$ be disjoint independent sets such that $|A|\ge|B|\ge1$, every vertex of $A$ is adjacent to every vertex outside $A$, 
and every vertex of $B$ is adjacent to every vertex outside $B$. 
If $a\in A$ and $b\in B$, then
\begin{equation*}
\inj(H,F-a)-\inj(H,F-b)\ge2e(H)(|A|-|B|)(1-C\beta)v(F)^{v(H)-2}.
\end{equation*}
\end{lemma}

\begin{proof}
The proof is essentially the same as the proof of \cite[Lemma 2.3]{Morrison2023}.
For the sake of completeness, we sketch the proof as follows.  

For convenience, we let $F_0=F-(A\cup B)$. For $S\subseteq V(H)$ and $X\in\{F-a,F-b\}$, let $\mathcal J_S(X)$ be the set of injective homomorphisms $\varphi:H\to X$ with $\varphi^{-1}(A\cup B)=S$. Since every edge between $A\cup B$ and $V(F_0)$ is present,
$$|\mathcal J_S(X)|=\inj(H[S],X[A\cup B])\inj(H-S,F_0).$$
Define $\Delta(S)=|\mathcal J_S(F-a)|-|\mathcal J_S(F-b)|$. Then $\inj(H,F-a)-\inj(H,F-b)=\sum_{S\subseteq V(H)}\Delta(S)$.

First we estimate the contribution of sets of size at most two. 
Since $A$ and $B$ are independent and $F$ is $\beta$-dense, $|A|,|B|\le\beta v(F)$, so $|V(F_0)|\ge(1-2\beta)v(F)$. 
Every vertex of $F_0$ has at most $\beta v(F)\le2\beta|V(F_0)|$ non-neighbors in $F_0$, and hence $F_0$ is $2\beta$-dense. 
By Lemma \ref{cub:lem:extension}, for every two element set $S$, we have $\inj(H-S,F_0)\ge(1-2\beta v(H)^2)|V(F_0)|^{v(H)-2}$. 
If $1-2\beta v(H)^2\le0$, then \eqref{cub:eq:f0-extension} below is immediate because its right-hand side is non-positive. 
If $1-2\beta v(H)^2>0$, then $|V(F_0)|\ge(1-2\beta)v(F)$ implies
$$\inj(H-S,F_0)\ge(1-2\beta v(H)^2)(1-2\beta)^{v(H)-2}v(F)^{v(H)-2}.$$
For non-negative $x$ and $0\le y\le1$, $(1-x)(1-y)^p\ge1-x-py$: 
if $x\ge1$ this is immediate, while for $0\le x\le1$ it follows from $1-(1-y)^p\le py$. 
Thus
\begin{equation}\label{cub:eq:f0-extension}
\inj(H-S,F_0)\ge(1-2\beta v(H)^2-2\beta(v(H)-2))v(F)^{v(H)-2}.
\end{equation}
If $S$ is independent, then $\Delta(S)=0$. If $S$ consists of the ends of an edge of $H$, then $H[S]=K_2$ and
$$\Delta(S)=2\big((|A|-1)|B|-|A|(|B|-1)\big)\inj(H-S,F_0)=2(|A|-|B|)\inj(H-S,F_0).$$
Summing \eqref{cub:eq:f0-extension} over the edges of $H$ gives
\begin{equation}\label{cub:eq:small-S}
\sum_{|S|\le2}\Delta(S)\ge2e(H)(|A|-|B|)(1-(2v(H)^2+2v(H)-4)\beta)v(F)^{v(H)-2}.
\end{equation}

It remains to control the possible negative contribution from $|S|\ge3$. 
Choose $A'\subseteq A$ with $a\in A'$ and $|A'|=|B|$, and choose a bijection between $A'$ and $B$ that pairs $a$ with $b$. 
Let $\iota$ swap each vertex of $A'$ with its paired vertex in $B$, and let $\iota$ fix every vertex of $F_0$. 
Consider an injective homomorphism $\varphi:H\to F-b$. 
If $a$ is not in its image, keep $\varphi$ unchanged. 
If $a$ is in its image and the image avoids $A\setminus A'$, replace $\varphi$ by $\iota\circ\varphi$. 
In the second case the new image avoids $a$, and the complete bipartite structure between $A'$ and $B$, 
together with the complete adjacency to $F_0$, shows that $\iota\circ\varphi$ is an injective homomorphism to $F-a$. 
The two cases have disjoint images according as $b$ is absent or present, and the construction is injective within each case.

A homomorphism not covered by this injection and belonging to some $\mathcal J_S(F-b)$ with $|S|\ge3$ must map one vertex of $H$ to $a$, 
a second vertex to $A\setminus A'$, and a third vertex to $A\cup B$. 
If we disregard all adjacency and injectivity restrictions, then the number of such maps is at most 
$$v(H)^3(|A|-|B|)(|A|+|B|)v(F)^{v(H)-3}\le2\beta v(H)^3(|A|-|B|)v(F)^{v(H)-2}.$$
It means that
\begin{equation}\label{cub:eq:large-S}
\sum_{|S|\ge3}\Delta(S)\ge-2\beta v(H)^3(|A|-|B|)v(F)^{v(H)-2}.
\end{equation}
Combining \eqref{cub:eq:small-S} and \eqref{cub:eq:large-S}, and using $e(H)\ge v(H)/2$, we conclude that
$$\begin{aligned}
\sum_S\Delta(S)
&\ge2e(H)(|A|-|B|)\left(1-\left(2v(H)^2+2v(H)-4+\frac{v(H)^3}{e(H)}\right)\beta\right)v(F)^{v(H)-2}\\
&\ge2e(H)(|A|-|B|)(1-C\beta)v(F)^{v(H)-2},
\end{aligned}$$
as wanted.
\end{proof}

\begin{lemma}\label{cub:lem:dense-rpartite}
Let $C=4v(H)^2+2v(H)-4$. If $0<\beta\le1/4$, $C\beta<1$, and $Q$ is an $n$-vertex $r$-partite $\beta$-dense graph, then
\begin{equation*}
\inj(H,T_r(n))-\inj(H,Q)\ge2e(H)(1-C\beta)(e(T_r(n))-e(Q))n^{v(H)-2}.
\end{equation*}
\end{lemma}

\begin{proof}
We induct on the non-negative integer $e(T_r(n))-e(Q)$. 
If $e(T_r(n))-e(Q)=0$, equality in Tur\'an's theorem~\cite{turan1941} implies $Q\cong T_r(n)$.
Then we assume that $e(T_r(n))-e(Q)\ge1$. 

We first suppose that a crossing edge is missing from an $r$-partition of $Q$, and let $Q'$ be obtained by adding one such edge. 
For each oriented edge of $H$, Lemma~\ref{cub:lem:extension} extends the map that sends its ends to the new oriented edge in at least $(1-\beta v(H)^2)n^{v(H)-2}$ ways. 
Different oriented edges produce disjoint sets of new injective homomorphisms. 
Hence $\inj(H,Q')-\inj(H,Q)\ge2e(H)(1-\beta v(H)^2)n^{v(H)-2}\ge2e(H)(1-C\beta)n^{v(H)-2}$. 
Applying the induction hypothesis to $Q'$ proves the lemma.

It remains to consider a complete but unbalanced $r$-partite graph. 
Choose a largest part of size $p$ and a smallest part of size $q$; then $p\ge q+2$. 
Add a new vertex to the smaller part and call the resulting $(n+1)$-vertex graph $F$. 
Let $a$ be a vertex of the part of size $p$ and let $b$ be the new vertex. 
Then $Q=F-b$, while $Q'=F-a$ is obtained from $Q$ by moving one vertex from the larger part to the smaller part. 
Since $Q$ is complete and $\beta$-dense, every part of $Q$ has size at most $\beta n$. 
The largest part of $F$ still has size $p$, and therefore every vertex of $F$ has degree at least $n+1-p\ge(1-\beta)(n+1)$. 
Thus $F$ is $\beta$-dense. 
It follows from Lemma~\ref{cub:lem:transfer} that
$$\inj(H,Q')-\inj(H,Q)\ge2e(H)(p-q-1)(1-C\beta)(n+1)^{v(H)-2}\ge2e(H)(p-q-1)(1-C\beta)n^{v(H)-2}.$$
Every part of $Q'$ has size at most $p\le\beta n$, so $Q'$ is also $\beta$-dense. 
Since $e(Q')-e(Q)=p-q-1$, the induction hypothesis applied to $Q'$ again proves this lemma.
\end{proof}

\begin{proposition}\label{cub:prop:parameters}
Let $H$ be a graph with $v(H)\ge2$, let $C=4v(H)^2+2v(H)-4$, and suppose $r\ge4v(H)^3+11v(H)^2$. Then there exist positive reals $\eta$ and $\beta$ such that
\begin{equation*}
2\eta-(v(H)-2)\eta^2>\frac{v(H)}{r}\quad \text{and} \quad \eta+\frac1r<\beta<\frac1{2C}.
\end{equation*}
\end{proposition}

\begin{proof}
Setting $x=1/(2C)-1/r$. 
Since $$4v(H)^3+11v(H)^2-2C=4v(H)^3+3v(H)^2-4v(H)+8>0,$$ we have $r>2C$ and $x>0$. 
We first show that $2x-(v(H)-2)x^2>v(H)/r$. 
Since $0<x<1/(2C)$, we have
$$2x-(v(H)-2)x^2=2x\left(1-\frac{v(H)-2}{2}x\right)\ge2x\left(1-\frac{v(H)-2}{4C}\right).$$
The last expression is larger than $v(H)/r$ when
\begin{equation*}
r>2C+\frac{4C^2v(H)}{4C-v(H)+2}.
\end{equation*}
Note that
\begin{eqnarray*}
&(4v(H)^3+11v(H)^2-2C)(4C-v(H)+2)-4C^2v(H)\\
=&12v(H)^4+13v(H)^3+122v(H)^2+48v(H)-112>0,
\end{eqnarray*}
where the last inequality holds for $v(H)\ge2$ because $12v(H)^4-112>0$ and all remaining terms are positive.
Thus $r\geq 4v(H)^3+11v(H)^2>2C+4C^2v(H)/(4C-v(H)+2)$.
It implies that $2x-(v(H)-2)x^2>v(H)/r$. 

By continuity, we may choose $0<\eta<x$ sufficiently close to $x$ that the first inequality holds. 
Since $\eta+1/r<1/(2C)$, a number $\beta$ satisfying the second inequalities also exists.
\end{proof}

For the rest of the proof of Theorem~\ref{thm:cubic}, 
fix $\eta$ and $\beta$ as in Proposition~\ref{cub:prop:parameters}, and let
\begin{equation}\label{cub:eq:sigma-lambda}
\sigma:=2\eta-(v(H)-2)\eta^2-\frac{v(H)}{r}>0\quad \text{and} \quad \lambda:=\frac12-C\beta>0.
\end{equation}

\section{Proof of Theorems~\ref{thm:cubic} and~\ref{thm:alln}}\label{sec:cubic}

We will repeatedly use the following  Zykov symmetrization. 
Delete a vertex $v$ and add a new vertex $v'$ whose neighborhood is $N_G(v_0)\setminus\{v\}$. 
The new vertex is not adjacent to $v_0$. 
If the original graph is $K_{r+1}$-free, then so is the new graph: a clique containing $v'$ cannot contain $v_0$, and replacing $v'$ by $v_0$ would give a clique of the same size in the original graph. 
Moreover, at most $v(H)(v(H)-1)(n-2)_{v(H)-2}\le v(H)^2n^{v(H)-2}$ injective homomorphisms contain both $v$ and $v_0$. Therefore, if the new graph is denoted by $G'$, then
\begin{equation}\label{cub:eq:symmetrization-gain}
\inj(H,G')-\inj(H,G)\ge I_G(v_0)-I_G(v)-v(H)^2n^{v(H)-2}.
\end{equation}

We divide this section into three parts: the proof of Tur\'an-goodness, Tur\'an-stability, and the proof of Theorem \ref{thm:alln}.
Throughout this section, we assume that \(H\) has no isolated vertices.

\subsection{Tur\'an-goodness}

\begin{proof}[Proof of the Tur\'an-goodness statement in Theorem~\ref{thm:cubic}]
Let $G$ be an $n$-vertex $K_{r+1}$-free graph maximizing $\inj(H,G)$. We first show that $G$ is $\eta$-dense when $n$ is sufficiently large.

Suppose that $d_G(v)\le(1-\eta)n$. 
Then $x_v\ge\eta_n:=(\eta n-1)/(n-1)$.
Let $v_0$ maximize $I_G(v_0)$. 
Combining Lemmas~\ref{cub:lem:vertex-upper} and \ref{cub:lem:turan-lower} with the monotonicity of $\Psi$, we have
$$I_G(v_0)-I_G(v)\ge e(H)\left(2\eta_n-(v(H)-2)\eta_n^2-\frac{v(H)n}{r(n-1)}\right)\Gamma_n.$$
The expression in parentheses tends to $\sigma$ as $n\to\infty$, and $\Gamma_n/n^{v(H)-1}\to1$. 
Hence, for all sufficiently large $n$, it is at least $\sigma/2$ and $v(H)^2n^{v(H)-2}\le e(H)\sigma \Gamma_n/4$. 
Let $G'$ be the graph obtained from $G$ by symmetrization.
By \eqref{cub:eq:symmetrization-gain}, $G'$ satisfies
$$\inj(H,G')-\inj(H,G)\ge I_G(v_0)-I_G(v)-v(H)^2n^{v(H)-2}\ge e(H)\sigma \Gamma_n/4>0,$$
contradicting the choice of $G$. Thus $G$ is $\eta$-dense.

Let $Q$ be a spanning $r$-partite subgraph of $G$ with the maximum number of edges. Theorem~\ref{cub:thm:furedi} implies
\begin{equation}\label{cub:eq:furedi-L}
e(G)-e(Q)\le e(T_r(n))-e(G).
\end{equation}
By Lemma~\ref{cub:lem:max-cut}, $d_Q(v)\ge(1-1/r)d_G(v)\ge(1-1/r)(1-\eta)n>(1-\beta)n$, 
so $Q$ is $\beta$-dense. 
By \eqref{cub:eq:furedi-L}, we conclude $e(T_r(n))-e(Q)=e(T_r(n))-e(G)+e(G)-e(Q)\ge2(e(G)-e(Q))$.

For each edge of $G-Q$ and each of the $2e(H)$ oriented edges of $H$, at most $(n-2)_{v(H)-2}\le n^{v(H)-2}$ injective maps send that oriented edge to the chosen host edge. 
A union bound shows that deleting the $e(G)-e(Q)$ edges of $G-Q$ destroys at most $2e(H)(e(G)-e(Q))n^{v(H)-2}$ injective homomorphisms. 
Therefore $\inj(H,Q)\ge\inj(H,G)-2e(H)(e(G)-e(Q))n^{v(H)-2}$. 
It follows from Lemma~\ref{cub:lem:dense-rpartite} that
\begin{eqnarray*}
&&\inj(H,T_r(n))-\inj(H,G)\\
&\ge&2e(H)\big((1-C\beta)(e(T_r(n))-e(Q))-(e(G)-e(Q))\big)n^{v(H)-2}\\
&\ge&2e(H)((e(T_r(n))-e(Q))/2-(e(G)-e(Q)))n^{v(H)-2}\\
&\ge&0.  
\end{eqnarray*}
Thus $T_r(n)$ is extremal for every sufficiently large $n$.
\end{proof}

\subsection{Tur\'an-stability}

We first relate the edge deficit of a maximum $r$-cut to edit distance.

\begin{lemma}\label{cub:lem:edit-deficit}
Let $G$ be an $n$-vertex $K_{r+1}$-free graph, and let $Q$ be a maximum spanning $r$-partite subgraph of $G$. Then
\begin{equation*}
\ed(G,T_r(n))\le \frac32\bigl(e(T_r(n))-e(Q)\bigr)+n\sqrt{r\bigl(e(T_r(n))-e(Q)\bigr)+\frac{r^2}{4}}.
\end{equation*}
\end{lemma}

\begin{proof}
Let $V_1,\ldots,V_r$ be the parts of $Q$, and let $P$ be the complete $r$-partite graph with these parts. 
By the maximality of $Q$, every edge of $G$ joining two different parts belongs to $Q$.
Hence $E(Q)=E(G)\cap E(P)$.
Theorem~\ref{cub:thm:furedi}, together with the maximality of $Q$, gives $e(G)-e(Q)\le e(T_r(n))-e(G)$.
So we have $e(G)-e(Q)\le \frac12\bigl(e(T_r(n))-e(Q)\bigr)$. 
Since $e(P)\le e(T_r(n))$, we also have $e(P)-e(Q)\le e(T_r(n))-e(Q)$. 
It follows that
\begin{equation}\label{cub:eq:G-to-P}
|E(G)\triangle E(P)|=e(G)-e(Q)+e(P)-e(Q)\le \frac32\bigl(e(T_r(n))-e(Q)\bigr).
\end{equation}

We write $x_i=|V_i|$ for $1\le i\le r$, and choose integers $y_i\in\{\lfloor n/r\rfloor,\lceil n/r\rceil\}$ such that $\sum_{i=1}^r y_i=n$. 
Since $\sum_{i=1}^r x_i=\sum_{i=1}^r y_i=n$, we deduce
$$2\bigl(e(T_r(n))-e(P)\bigr)=\sum_{i=1}^r x_i^2-\sum_{i=1}^r y_i^2=\sum_{i=1}^r\left(x_i-\frac nr\right)^2-\sum_{i=1}^r\left(y_i-\frac nr\right)^2.$$
Let $n=ar+b$, where $0\le b<r$. 
We obtain that $\sum_{i=1}^r(y_i-n/r)^2=b(1-b/r)^2+(r-b)(b/r)^2=b(r-b)/r\le r/4$. 
Hence
$$\sum_{i=1}^r(x_i-y_i)^2\le 2\sum_{i=1}^r\left(x_i-\frac nr\right)^2+2\sum_{i=1}^r\left(y_i-\frac nr\right)^2\le 4\bigl(e(T_r(n))-e(P)\bigr)+r.$$

A balanced partition with part sizes $y_1,\ldots,y_r$ can be obtained by moving exactly $\frac12\sum_{i=1}^r|x_i-y_i|$ vertices. By the Cauchy--Schwarz inequality,
$$\left(\frac12\sum_{i=1}^r|x_i-y_i|\right)^2\le \frac r4\sum_{i=1}^r(x_i-y_i)^2\le r\bigl(e(T_r(n))-e(P)\bigr)+\frac{r^2}{4}\le r\bigl(e(T_r(n))-e(Q)\bigr)+\frac{r^2}{4}.$$
Thus $P$ can be changed into a copy of $T_r(n)$ by moving at most $\sqrt{r\bigl(e(T_r(n))-e(Q)\bigr)+r^2/4}$ vertices. 
Every pair whose edge status changes has at least one moved endpoint, 
so this changes at most $n\sqrt{r\bigl(e(T_r(n))-e(Q)\bigr)+r^2/4}$ edges. 
Combining this bound with \eqref{cub:eq:G-to-P} establishes the expected result.
\end{proof}

\begin{proof}[Proof of the stability statement in Theorem~\ref{thm:cubic}]
Fix $0<\varepsilon\le 1$. The Tur\'an-goodness statement has already been proved. Hence there exists $n_0$ such that
$\ex(n,H,K_{r+1})=\N(H,T_r(n))$ for every $n\ge n_0$.

Choose $\delta_0>0$ such that
$$\delta_0\le \min\left\{\frac{e(H)\sigma}{16v(H)},\frac{\varepsilon e(H)\sigma}{8},\frac{e(H)\lambda\varepsilon^2}{400r}\right\},$$
and set $\delta=\delta_0/|\operatorname{Aut}(H)|$. We take $n$ sufficiently large, in particular $n\ge n_0$.
Let $G$ be an $n$-vertex $K_{r+1}$-free graph satisfying
$\N(H,G)\ge \ex(n,H,K_{r+1})-\delta n^{v(H)}.$
Since $\inj(H,F)=|\operatorname{Aut}(H)|\N(H,F)$ for every graph $F$, we obtain
\begin{equation*}
\inj(H,G)
\ge |\operatorname{Aut}(H)|\ex(n,H,K_{r+1})-|\operatorname{Aut}(H)|\delta n^{v(H)}
=\inj(H,T_r(n))-\delta_0n^{v(H)}.
\end{equation*}

We now use the symmetrization process. Suppose that the current graph $G'$ satisfies
$\inj(H,G')\ge\inj(H,T_r(n))-\delta_0n^{v(H)}$ and is not $\eta$-dense. Then there is a vertex $v$ such that $d_{G'}(v)\le(1-\eta)n$. Choose a vertex $v_0$ for which $I_{G'}(v_0)$ is maximum, and symmetrize $v$ to $v_0$. Applying Lemma~\ref{cub:lem:vertex-upper}, Lemma~\ref{cub:lem:turan-lower}, and \eqref{cub:eq:symmetrization-gain}, the increase in the number of injective homomorphisms is at least
$$e(H)\left(2\eta_n-(v(H)-2)\eta_n^2-\frac{v(H)n}{r(n-1)}\right)\Gamma_n-v(H)\delta_0n^{v(H)-1}-v(H)^2n^{v(H)-2},$$
where $\eta_n=(\eta n-1)/(n-1)$.

Taking $n$ sufficiently large such that
$$2\eta_n-(v(H)-2)\eta_n^2-\frac{v(H)n}{r(n-1)}\ge\frac{3\sigma}{4},
\qquad
\Gamma_n\ge\frac12n^{v(H)-1},$$
and $v(H)^2n^{v(H)-2}\le e(H)\sigma\Gamma_n/8$. Since $\delta_0\le e(H)\sigma/(16v(H))$, we also have
$v(H)\delta_0n^{v(H)-1}\le e(H)\sigma\Gamma_n/8$. Therefore every symmetrization step increases the number of injective homomorphisms by at least
\begin{equation}\label{cub:eq:symmetrization-step-gain}
\frac{e(H)\sigma}{4}n^{v(H)-1}.
\end{equation}
In particular, the new graph still satisfies the same lower bound on its number of injective homomorphisms. Thus the process can be continued until an $\eta$-dense graph is obtained.

Every graph obtained during the process is $K_{r+1}$-free. Since $n\ge n_0$, the Tur\'an-goodness statement shows that its number of injective homomorphisms is at most $\inj(H,T_r(n))$. The initial difference from this upper bound is at most $\delta_0n^{v(H)}$. It follows from \eqref{cub:eq:symmetrization-step-gain} that the process has at most $4\delta_0n/(e(H)\sigma)$ steps. Each step changes at most $n$ adjacencies. Hence the resulting $\eta$-dense graph $G^*$ satisfies $\inj(H,G^*)\ge\inj(H,G)$ and
$$\ed(G^*,T_r(n))\ge \ed(G,T_r(n))-\frac{4\delta_0}{e(H)\sigma}n^2.$$
Since $\delta_0\le\varepsilon e(H)\sigma/8$, if $\ed(G,T_r(n))>\varepsilon n^2$, then
$\ed(G^*,T_r(n))>\varepsilon n^2/2$.

Suppose, for the sake of contradiction, that $\ed(G,T_r(n))>\varepsilon n^2$. Let $Q$ be a spanning $r$-partite subgraph of $G^*$ with the maximum number of edges. By Lemma~\ref{cub:lem:max-cut}, the graph $Q$ is $\beta$-dense. We claim that
\begin{equation}\label{cub:eq:partite-deficit-lower}
e(T_r(n))-e(Q)\ge \frac{\varepsilon^2}{400r}n^2.
\end{equation}
Suppose that this inequality is false. Lemma~\ref{cub:lem:edit-deficit} and the inequality $\sqrt{a+b}\le\sqrt a+\sqrt b$ give
$$\ed(G^*,T_r(n))<
\frac{3\varepsilon^2}{800r}n^2+\frac{\varepsilon}{20}n^2+\frac r2n
<\frac{\varepsilon}{2}n^2$$
for $n\ge10r/\varepsilon$. This contradicts $\ed(G^*,T_r(n))>\varepsilon n^2/2$, and proves \eqref{cub:eq:partite-deficit-lower}.

Since $Q$ has the maximum number of edges among all spanning $r$-partite subgraphs of $G^*$, Theorem~\ref{cub:thm:furedi} gives
$e(G^*)-e(Q)\le e(T_r(n))-e(G^*).$
Consequently, we obtain
$$e(T_r(n))-e(Q)
=e(T_r(n))-e(G^*)+e(G^*)-e(Q)
\ge2\bigl(e(G^*)-e(Q)\bigr).$$

Deleting the $e(G^*)-e(Q)$ edges in $E(G^*)\setminus E(Q)$ destroys at most
$2e(H)\bigl(e(G^*)-e(Q)\bigr)n^{v(H)-2}$ injective homomorphisms, because each deleted edge can be the image of at most $2e(H)n^{v(H)-2}$ injective homomorphisms from $H$. Hence
$$\inj(H,G^*)-\inj(H,Q)
\le2e(H)\bigl(e(G^*)-e(Q)\bigr)n^{v(H)-2}.$$
It follows from Lemma~\ref{cub:lem:dense-rpartite} that
\begin{equation*}
\inj(H,T_r(n))-\inj(H,G^*)\ge2e(H)\left((1-C\beta)\bigl(e(T_r(n))-e(Q)\bigr)-\bigl(e(G^*)-e(Q)\bigr)\right)n^{v(H)-2}.
\end{equation*}
By the preceding edge inequality and \eqref{cub:eq:sigma-lambda}, the expression in parentheses is at least
$\lambda\bigl(e(T_r(n))-e(Q)\bigr)$. 
Therefore, by \eqref{cub:eq:partite-deficit-lower}, we get
$$\inj(H,T_r(n))-\inj(H,G^*)
\ge2e(H)\lambda\bigl(e(T_r(n))-e(Q)\bigr)n^{v(H)-2}
\ge\frac{e(H)\lambda\varepsilon^2}{200r}n^{v(H)}.$$

On the other hand, $\inj(H,G^*)\ge\inj(H,G)$, and thus
$$\inj(H,T_r(n))-\inj(H,G^*)\le\delta_0n^{v(H)}.$$
This contradicts $\delta_0\le e(H)\lambda\varepsilon^2/(400r)$. Therefore every $n$-vertex $K_{r+1}$-free graph $G$ satisfying
$\N(H,G)\ge\ex(n,H,K_{r+1})-\delta n^{v(H)}$ has
$\ed(G,T_r(n))\le\varepsilon n^2$. This proves that $H$ is $K_{r+1}$-Tur\'an-stable.
\end{proof}

\subsection{Proof of Theorem~\ref{thm:alln}}
\begin{proof}[Proof of Theorem~\ref{thm:alln}]
The statement is immediate when $n\le r$, since $T_r(n)=K_n$. 
Assume that $n\ge r+1$ and $r\ge 40v(H)^3$, and let $G$ maximize $\inj(H,G)$ among all $n$-vertex $K_{r+1}$-free graphs.

Suppose that $d_G(v)\le (1-4v(H)/r)n$ for some vertex $v$. 
Then $x_v\ge (4v(H)-1)/r$, because $(4v(H)n/r-1)/(n-1)\ge (4v(H)-1)/r$ for $n\ge r+1$. 
Choose a vertex $v_0$ for which $I_G(v_0)$ is maximum, and let $G'$ be obtained by symmetrizing $v$ to $v_0$. As above,
$$I_G(v_0)-I_G(v)\ge e(H)\left(\frac{2(4v(H)-1)}{r}-\frac{(v(H)-2)(4v(H)-1)^2}{r^2}-\frac{v(H)n}{r(n-1)}\right)\Gamma_n.$$
Since $n/(n-1)\le 1+1/r$ and $(v(H)-2)(4v(H)-1)^2+v(H)\le 16v(H)^3$, the expression in parentheses is at least
$$\frac{7v(H)-2}{r}-\frac{(v(H)-2)(4v(H)-1)^2+v(H)}{r^2}\ge \frac{7v(H)-2-2/5}{r}\ge \frac{5v(H)}{r}.$$
Moreover,
$$\Gamma_n=n^{v(H)-1}\prod_{i=1}^{v(H)-1}\left(1-\frac{i}{n}\right)\ge n^{v(H)-1}\left(1-\frac{v(H)(v(H)-1)}{2n}\right)\ge \frac{159}{160}n^{v(H)-1},$$
where we based on $n\ge r+1\ge 40v(H)^3$ and the inequality $\prod_i(1-z_i)\ge 1-\sum_i z_i$ for $z_i\in[0,1]$. 
Since $e(H)\ge v(H)/2$, we deduce from \eqref{cub:eq:symmetrization-gain} that
$$\inj(H,G')-\inj(H,G)\ge \left(\frac{795}{320}\frac{n}{r}-1\right)v(H)^2n^{v(H)-2}>0,$$
contradicting the choice of $G$. Therefore $G$ is $(4v(H)/r)$-dense.

Let $Q$ be a maximum spanning $r$-partite subgraph of $G$. 
Using Lemma~\ref{cub:lem:max-cut}, the graph $Q$ is $(4v(H)+1)/r$-dense. Clearly, $0<(4v(H)+1)/r\le 1/4$. We also have $C(4v(H)+1)/r<1/2$. Indeed,
$$20v(H)^3-C(4v(H)+1)=4v(H)^3-12v(H)^2+14v(H)+4>0.$$
For $v(H)=2$, this is immediate, while for $v(H)\ge 3$, the expression on the right equals $4v(H)^2(v(H)-3)+14v(H)+4>0$. Hence $C(4v(H)+1)<20v(H)^3\le r/2$.

By Theorem~\ref{cub:thm:furedi}, $e(G)-e(Q)\le e(T_r(n))-e(G)$, and therefore
$$e(T_r(n))-e(Q)=e(T_r(n))-e(G)+e(G)-e(Q)\ge 2\bigl(e(G)-e(Q)\bigr).$$
Deleting the edges in $E(G)\setminus E(Q)$ destroys at most $2e(H)(e(G)-e(Q))n^{v(H)-2}$ injective homomorphisms, 
because for each deleted edge and each of the $2e(H)$ oriented edges of $H$, 
there are at most $n^{v(H)-2}$ injective maps sending that oriented edge to the deleted edge. 
We obtain from Lemma~\ref{cub:lem:dense-rpartite} that
\begin{eqnarray*}
&&\inj(H,T_r(n))-\inj(H,G)\\
&\ge& 2e(H)\left(\left(1-\frac{C(4v(H)+1)}{r}\right)\bigl(e(T_r(n))-e(Q)\bigr)-\bigl(e(G)-e(Q)\bigr)\right)n^{v(H)-2}.
\end{eqnarray*}
Since $C(4v(H)+1)/r<1/2$ and $e(T_r(n))-e(Q)\ge 2(e(G)-e(Q))$, we get
$$\inj(H,T_r(n))-\inj(H,G)\ge 2e(H)\left(\frac12\bigl(e(T_r(n))-e(Q)\bigr)-\bigl(e(G)-e(Q)\bigr)\right)n^{v(H)-2}\ge 0.$$
Thus $\inj(H,G)\le \inj(H,T_r(n))$ for every $n$, as required.
\end{proof}

\section{Proof of Theorem~\ref{thm:counterexamples}}\label{sec:counterexamples}

In this section, we will prove Theorem~\ref{thm:counterexamples}. 
We first show that $H_{r,s,\ell}$ is $K_r$-Tur\'an-good. 
We then determine its proper $(r-1)$-colorings and proper $r$-colorings exactly and use these counts to prove that it is not $K_{r+1}$-Tur\'an-good.

\subsection{Properties of $H_{r,s,\ell}$}

Two copies of $K_q$ in a graph are called adjacent if they have exactly $q-1$ common vertices. The $q$-skeleton of a connected component is connected if, for any two vertices $x$ and $y$ in that component, there are copies $Q_1,\ldots,Q_m$ of $K_q$ such that $x\in V(Q_1)$, $y\in V(Q_m)$, and $Q_i$ and $Q_{i+1}$ are adjacent for every $1\leq i<m$.

\begin{theorem}[Gy\H{o}ri, Pach, and Simonovits \cite{Gyori1991}]\label{ce:thm:Gyori1991}
Let $q\geq2$, and let $F$ be a $q$-partite graph on $m\geq q$ vertices. Suppose that $F$ contains $\lfloor m/q\rfloor$ vertex-disjoint copies of $K_q$, and that the $q$-skeleton of every connected component of $F$ is connected. Then, for every $n$-vertex $K_{q+1}$-free graph $X$, we have $\mathcal N(F,X)\leq\mathcal N(F,T_q(n))$.
\end{theorem}

\begin{lemma}\label{ce:lem:H-good}
Let $r\geq3$, $s\geq4$, and $0\leq\ell\leq s$. Then $H_{r,s,\ell}$ is $K_r$-Tur\'an-good.
\end{lemma}

\begin{proof}
Let $q=r-1$, and let the parts of $H_{r,s,\ell}$ be $A_1,\ldots,A_q$. 
Without loss of generality, assume that $A_1=\{a_{1,1},\ldots,a_{1,s}\}$ and $A_2=\{a_{2,1},\ldots,a_{2,s}\}$ so that the deleted edges are $a_{1,i}a_{2,i}$ for $1\leq i\leq\ell$. 
For $3\leq j\leq q$, we write $A_j=\{a_{j,1},\ldots,a_{j,s}\}$.

First we will show that the graph contains $s$ vertex-disjoint copies of $K_q$. 
Choose a permutation $\pi$ of $\{1,\ldots,s\}$ satisfying $\pi(i)\neq i$ for every $1\leq i\leq\ell$. 
For $\ell=0$ take the identity permutation; for $\ell=1$ interchange $1$ and $2$; and for $\ell\geq2$ cyclically permute $1,\ldots,\ell$ and fix the other indices. 
For each $1\leq i\leq s$, the set $\{a_{1,i},a_{2,\pi(i)},a_{3,i},\ldots,a_{q,i}\}$ induces a copy of $K_q$, where the last list is absent when $q=2$. 
The only edge that could be missing is the edge between the selected vertices of $A_1$ and $A_2$, and that edge is present by the definition of $\pi$. 
These $s$ copies are vertex-disjoint. 
Since $H_{r,s,\ell}$ has $qs$ vertices, their number is exactly $\lfloor qs/q\rfloor$.

We next prove that the $q$-skeleton is connected. 
Fix an edge $uv$ with $u\in A_1$ and $v\in A_2$, and fix one vertex in each of $A_3,\ldots,A_q$; 
together they form a copy $Q^*$ of $K_q$. Let $Q$ be any copy of $K_q$, and let $x\in A_1$ and $y\in A_2$ be its vertices in the first two parts. 
Each of $x$ and $u$ has at most one nonneighbor in $A_2$. 
Hence at most two vertices of $A_2$ fail to be adjacent to at least one of $x$ and $u$. 
Since $s\geq4$, there is a vertex $z\in A_2$ adjacent to both $x$ and $u$. 
Starting from $Q$, replace $y$ by $z$, then replace $x$ by $u$, and finally replace $z$ by $v$, omitting a replacement that does not change the set. 
Every intermediate set is a copy of $K_q$, and each genuine replacement changes exactly one vertex. 
We may then replace the vertices in $A_3,\ldots,A_q$ one at a time by the corresponding vertices of $Q^*$. 
Thus every copy of $K_q$ can be joined to $Q^*$ by adjacent copies of $K_q$.

Every vertex belongs to a copy of $K_q$. 
A vertex of $A_1$ has a neighbor in $A_2$, and a vertex of $A_2$ has a neighbor in $A_1$; 
after choosing such an edge, choose one vertex from every other part. 
A vertex in $A_j$ with $j\geq3$ can be combined with the fixed edge $uv$ and one vertex from each remaining part. 
Consequently, any two vertices can first be placed in copies of $K_q$ and then joined through $Q^*$. The $q$-skeleton is connected.

All assumptions of Theorem~\ref{ce:thm:Gyori1991} hold, so $T_q(n)=T_{r-1}(n)$ maximizes the number of copies of $H_{r,s,\ell}$ among all $K_{q+1}=K_r$-free graphs. Hence $H_{r,s,\ell}$ is $K_r$-Tur\'an-good.
\end{proof}

For a graph $F$, let $P_F(k)$ be the number of proper vertex colorings of $F$ with the labeled colors $1,\ldots,k$; colors are allowed to be unused.

\begin{lemma}\label{ce:lem:asymptotic}
Let $k$ be fixed. 
As $n$ tends to infinity through multiples of $k$,
$$\mathrm{inj}(F,T_k(n))=\frac{P_F(k)}{k^{v(F)}}n^{v(F)}+O(n^{{v(F)}-1}).$$
\end{lemma}

\begin{proof}
Suppose  the $k$ parts of $T_k(n)$ by $1,\ldots,k$. 
Every injective homomorphism from $F$ to $T_k(n)$ gives a proper $k$-coloring of $F$ by recording the part containing the image of each vertex. 
Conversely, fix a proper coloring whose color class sizes are $m_1,\ldots,m_k$. 
The number of injective homomorphisms inducing this coloring is $\prod_{i=1}^k(n/k)_{m_i}$, where $(x)_a=x(x-1)\cdots(x-a+1)$ and $(x)_0=1$. 
Since $m_1+\cdots+m_k=v(F)$, this product is $n^{v(F)}/k^{v(F)}+O(n^{{v(F)}-1})$. 
Summing over all $P_F(k)$ proper colorings proves the lemma.
\end{proof}

\begin{corollary}\label{ce:cor:criterion}
Let $q\geq2$, and let $F$ be $K_{q+1}$-Tur\'an-good. If
$$\frac{P_F(q)}{q^{v(F)}}>\frac{P_F(q+1)}{(q+1)^{v(F)}},$$
then $F$ is not $K_{q+2}$-Tur\'an-good.
\end{corollary}

\begin{proof}
By Lemma~\ref{ce:lem:asymptotic}, the displayed strict inequality implies $\mathrm{inj}(F,T_q(n))>\mathrm{inj}(F,T_{q+1}(n))$ for all sufficiently large $n$ divisible by $q(q+1)$. Hence $T_q(n)$ contains more copies of $F$ than $T_{q+1}(n)$. Since $T_q(n)$ is $K_{q+2}$-free, $T_{q+1}(n)$ fails to maximize the number of copies of $F$ among $K_{q+2}$-free graphs for infinitely many arbitrarily large $n$. Thus $F$ is not $K_{q+2}$-Tur\'an-good.
\end{proof}

\begin{lemma}\label{ce:lem:H-color}
Let $q\geq2$, $s\geq4$, and $0\leq\ell\leq s$. For $H=H_{q+1,s,\ell}$,
$$P_H(q)=q!\quad \text{and} \quad P_H(q+1)=q!\left[(q+1)(\ell+1)+\binom{q+1}{2}(2^s-2)\right].$$
\end{lemma}

\begin{proof}
Let the deleted matching be $\{x_iy_i:1\leq i\leq\ell\}$, where $x_i\in A_1$ and $y_i\in A_2$. 
We first note that after deleting the endpoints of at most two matching edges, 
the remaining graph still contains $K_q$. 
Indeed, each of $A_1$ and $A_2$ then has at least $s-2\geq2$ remaining vertices. 
A remaining vertex of $A_1$ has at most one nonneighbor in $A_2$, so it has a remaining neighbor there; choose one vertex from every other part.

Consider a proper $q$-coloring. 
Suppose, for the sake of contradiction, that $x_i$ and $y_i$ have the same color. 
This color cannot occur on any other vertex: every vertex of $A_1\setminus\{x_i\}$ is adjacent to $y_i$, every vertex of $A_2\setminus\{y_i\}$ is adjacent to $x_i$, and every vertex in the other parts is adjacent to both. 
After removing $x_i$ and $y_i$, a graph containing $K_q$ would be colored with at most $q-1$ colors, a contradiction. 
Thus no deleted pair is monochromatic. Every color class is therefore contained in one original part. 
The $q$ nonempty parts use pairwise disjoint nonempty sets of colors, and only $q$ colors are available, so each part receives one color. Hence $P_H(q)=q!$.

Now consider proper $(q+1)$-colorings. 
If no deleted pair is monochromatic, the coloring is a proper coloring of the complete $q$-partite graph with the same parts. 
Colorings using exactly $q$ colors contribute $(q+1)q!$. 
If all $q+1$ colors are used, exactly one part uses two colors. 
Choosing this part, the two colors, a bijection from the remaining colors to the remaining parts, and a two-coloring of the selected part that uses both colors gives $q\binom{q+1}{2}(q-1)!(2^s-2)=q!\binom{q+1}{2}(2^s-2)$ colorings.

It remains to count colorings with a monochromatic deleted pair. 
There cannot be two such pairs. 
Their common colors would be distinct, each common color could occur only on its corresponding pair, and after removing the four endpoints the remaining graph would contain $K_q$ but use at most $q-1$ colors. 
Thus exactly one deleted pair is monochromatic. 
Choose it in $\ell$ ways and choose its common color in $q+1$ ways. 
This color is used nowhere else. 
No remaining deleted pair can be monochromatic, by the same argument, and the other $q$ colors must consequently be assigned bijectively to the $q$ original parts. 
This obtain $\ell(q+1)q!$ colorings. Adding all cases proves the formula.
\end{proof}

\subsection{Proof of Theorem \ref{thm:counterexamples}}

\begin{lemma}\label{ce:lem:growth}
If $r\geq3$ and $s\geq12\lceil\log_2r\rceil$, then
$$\left(\frac r{r-1}\right)^{(r-1)s}>r^2 2^s.$$
\end{lemma}

\begin{proof}
By the binomial theorem,
$$\left(\frac r{r-1}\right)^{r-1}=\left(1+\frac1{r-1}\right)^{r-1}\geq1+1+\binom{r-1}{2}\frac1{(r-1)^2}\geq\frac94.$$
Since $(9/8)^6=531441/262144>2$ and $s\geq12\lceil\log_2r\rceil$,
$$\left(\frac r{r-1}\right)^{(r-1)s}
\geq\left(\frac94\right)^s
=2^s\left(\frac98\right)^s\geq2^s\left(\frac98\right)^{12\lceil\log_2r\rceil}
>2^s2^{2\lceil\log_2r\rceil}\geq r^22^s,$$
as required.
\end{proof}

\begin{proof}[Proof of Theorem~\ref{thm:counterexamples}]
Let $q=r-1$. 
Lemma~\ref{ce:lem:H-good} shows that every graph in the theorem is $K_r=K_{q+1}$-Tur\'an-good. 
In view of Corollary~\ref{ce:cor:criterion}, it remains to verify that
$$\frac{P_{H_{r,s,\ell}}(q)}{q^{qs}}>\frac{P_{H_{r,s,\ell}}(q+1)}{(q+1)^{qs}}.$$
By Lemma~\ref{ce:lem:H-color}, it is enough to prove
$$\left(\frac r{r-1}\right)^{(r-1)s}>r(\ell+1)+\binom r2(2^s-2).$$
Since $s\ge12\lceil\log_2r\rceil\ge24$, we have $s+1<2^s$. 
Thus, for $0\le\ell\le s$, we have $\ell+1\leq s+1<2^s$.
Therefore, by Lemma~\ref{ce:lem:growth}, we obtain
$$r(\ell+1)+\binom r2(2^s-2)<\left(r+\binom r2\right)2^s=\frac{r(r+1)}2\,2^s<r^22^s<\left(\frac r{r-1}\right)^{(r-1)s}.$$
This completes the proof.
\end{proof}

\section{Proof of Theorem \ref{thm:claw-stable}}\label{sec:claw}

Let $K_{n_1,\ldots,n_k}$ denote the complete $k$-partite graph with parts of sizes $n_1,\ldots,n_k$.
An edge of a graph $F$ is color-critical if deleting it lowers $\chi(F)$.
We use the following statement. 
It appears in the work of Hei, Hou and Liu~\cite{HeiHouLiu} and is recorded in the form needed here by Gerbner \cite{Gerbner2023}.

\begin{theorem}\label{claw:thm:stability-exactness}
Let $F$ have a color-critical edge, let $\chi(F)=k+1$, and let $\chi(H)\le k$. Suppose that, for every sufficiently large $n$, there is an $n$-vertex $F$-free graph $G$ satisfying $\N(H,G)=\ex(n,H,F)$ such that $G$ can be transformed into a complete $k$-partite graph by adding and deleting $o(n^2)$ edges. If
$$\N(H,K_{n_1,\ldots,n_k})\le \N(H,T_k(n))$$
for every complete $k$-partite graph of order $n$, then $H$ is $F$-Tur\'an-good.
\end{theorem}

The next theorem will be used to show that an unbalanced complete multipartite graph cannot give a counterexample (for Question \ref{q:monotonicity}) for a claw-free graph.

\begin{theorem}\label{thm:claw-balance}
Let $H$ be a claw-free graph, let $k\ge2$, and let $n_1,\ldots,n_k$ be nonnegative integers with $n_1+\cdots+n_k=n$. Then
$$\inj(H,K_{n_1,\ldots,n_k})\le\inj(H,T_k(n)).$$
Equivalently, $\N(H,K_{n_1,\ldots,n_k})\le\N(H,T_k(n))$.
\end{theorem}

Before we move to the technical proof of Theorem~\ref{thm:claw-balance}, 
we first show how this theorem  will be applied to complete the proof of Theorem~\ref{thm:claw-stable}.

\begin{proof}[Proof of Theorem~\ref{thm:claw-stable}]
Note that $F=K_{k+1}$ has a color-critical edge.
Therefore, the desired conclusion follows directly from Theorems~\ref{claw:thm:stability-exactness} and~\ref{thm:claw-balance}.
\end{proof}

\subsection{Proof of Theorem \ref{thm:claw-balance}}

The proof of Theorem~\ref{thm:claw-balance} first treats two host parts. 
We then fix the images outside two chosen parts and balance those two parts one vertex at a time.

We use the following form of Vandermonde's identity for falling factorials:
\begin{equation}\label{claw:eq:vandermonde}
\sum_{j=0}^{c}\binom{c}{j}\fall{u}{j}\fall{v}{c-j}=\fall{u+v}{c}.
\end{equation}
%It follows, for example, by counting injections from a $c$-element set into the disjoint union of sets of sizes $u$ and $v$, according to the number of elements mapped into the first set.

\begin{lemma}\label{claw:lem:induced-bipartite}
If $H$ is claw-free and $S\subseteq V(H)$ is such that $H[S]$ is bipartite, then every component of $H[S]$ is a path or an even cycle.
\end{lemma}

\begin{proof}
Suppose that some vertex $v\in V(H[S])$ has three distinct neighbors $x,y,z$ in $H[S]$. 
Since $H[S]$ is bipartite, the vertices $x,y,z$ belong to the same side of a bipartition and are therefore pairwise nonadjacent in $H[S]$. 
As $H[S]$ is induced in $H$, they are also pairwise nonadjacent in $H$. 
Hence $H[\{v,x,y,z\}]$ is an induced $K_{1,3}$, contradicting the assumption that $H$ is claw-free. 
Thus $\Delta(H[S])\le 2$. Every finite connected graph of maximum degree at most two is a path or a cycle, and every cycle in the bipartite graph $H[S]$ is even.
\end{proof}

\begin{lemma}\label{claw:lem:bipartite-formula}
Let $B$ be a bipartite graph whose components are paths or even cycles. For each component choose a bipartition and let $a$ be the sum of the smaller side sizes. Let $c$ be the number of components whose two side sizes are unequal, and let $e$ be the number of components whose two side sizes are equal. Then, for all nonnegative integers $x,y$,
\begin{equation}\label{claw:eq:formula}
\inj(B,K_{x,y})=2^e\fall{x}{a}\fall{y}{a}\fall{x+y-2a}{c}.
\end{equation}
\end{lemma}

\begin{proof}
For every component of $B$, the two sides of a bipartition have equal sizes or differ by one. This is immediate for paths and even cycles, including an isolated vertex, whose side sizes are $1$ and $0$.

Such a component can be placed in $K_{x,y}$ in two ways, obtained by swapping its two sides. Since the two sides have the same size, both choices put the same number of vertices in each part of $K_{x,y}$.
These components therefore contribute a factor $2^e$. 
Each of the remaining $c$ components has side sizes $b_i$ and $b_i+1$ for some $b_i\ge 0$. 
If exactly $j$ of these components send their larger side to the $x$-part, then the total numbers of vertices sent to the two host parts are $a+j$ and $a+c-j$. 
There are $\binom{c}{j}$ choices for those components. 
Once the orientations have been fixed, the vertices mapped into each host part may be injected independently. 
Consequently,
$$\inj(B,K_{x,y})=2^e\sum_{j=0}^{c}\binom{c}{j}\fall{x}{a+j}\fall{y}{a+c-j}.$$
Using $\fall{x}{a+j}=\fall{x}{a}\fall{x-a}{j}$ and the analogous identity for $y$, 
and then applying Vandermonde's identity \eqref{claw:eq:vandermonde}, we obtain
$$\inj(B,K_{x,y})=2^e\fall{x}{a}\fall{y}{a}\sum_{j=0}^{c}\binom{c}{j}\fall{x-a}{j}\fall{y-a}{c-j}=2^e\fall{x}{a}\fall{y}{a}\fall{x+y-2a}{c},$$
as expected.
\end{proof}

\begin{lemma}\label{claw:lem:two-part}
Let $H$ be claw-free and let $x,y$ be nonnegative integers with $y\ge x+2$. Then
$$\inj(H,K_{x,y})\le \inj(H,K_{x+1,y-1}).$$
\end{lemma}

\begin{proof}
If $H$ is not a bipartite graph, neither $K_{x,y}$ nor $K_{x+1,y-1}$ contains a copy of $H$. 
Hence both sides of the required inequality are zero. 
We suppose that $H$ is bipartite. 
By Lemma~\ref{claw:lem:induced-bipartite}, every component of $H$ is a path or an even cycle.

Consider the formula in Lemma~\ref{claw:lem:bipartite-formula}. 
Replacing $(x,y)$ by $(x+1,y-1)$ does not change $x+y$. 
Therefore, the factor $2^e\fall{x+y-2a}{c}$ in \eqref{claw:eq:formula} remains unchanged. 
Thus it is enough to prove
$$\fall{x+1}{a}\fall{y-1}{a}\ge \fall{x}{a}\fall{y}{a}.$$

If $\fall{x}{a}\fall{y}{a}=0$, the inequality follows immediately because the product on the left is nonnegative. We may therefore suppose that $x\ge a$ and $y\ge a$. If $a=0$, both products are equal to one. Suppose that $a\ge1$. Since all factors in the denominator below are positive, we may compare the two products by taking their ratio:
$$\frac{\fall{x+1}{a}\fall{y-1}{a}}{\fall{x}{a}\fall{y}{a}}=\frac{x+1}{x+1-a}\frac{y-a}{y}.$$
This ratio is at least one if and only if $(x+1)(y-a)\ge y(x+1-a)$. The difference between the two sides is $a(y-x-1)$, which is positive because $y\ge x+2$. Hence $\fall{x+1}{a}\fall{y-1}{a}\ge \fall{x}{a}\fall{y}{a}$. Substituting this inequality into \eqref{claw:eq:formula} completes the proof.
\end{proof}

We now prove the main theorem by balancing two host parts at a time.

\begin{proof}[Proof of Theorem~\ref{thm:claw-balance}]
It is enough to prove that if two part sizes satisfy $n_i\ge n_j+2$, 
then moving one vertex from the $i$th part to the $j$th part does not decrease the number of injective homomorphisms from $H$.

Let $X$ and $Y$ be the $i$th and $j$th parts, with $|X|=x$ and $|Y|=y$, where $x\ge y+2$. 
Let $R$ be the union of the remaining $k-2$ parts. 
We partition all injective homomorphisms from $H$ into the host graph according to the set $S=\varphi^{-1}(X\cup Y)$
and the restriction $\psi=\varphi|_{V(H)\setminus S}$.

Fix a set $S\subseteq V(H)$ and an injective homomorphism $\psi$ from $H-S$ into the complete $(k-2)$-partite graph induced by $R$. 
Every extension of $\psi$ that maps precisely the vertices of $S$ into $X\cup Y$ is obtained by an injective homomorphism from the induced graph $H[S]$ into $K_{x,y}$. 
Indeed, the edges with both endpoints in $S$ are respected exactly when the restriction to $S$ is such a homomorphism. 
Every edge having one endpoint in $S$ and one in $V(H)\setminus S$ is automatically respected because $X\cup Y$ is complete to every part contained in $R$. 
There are no further conditions, since graph homomorphisms do not have to preserve nonedges. 
Hence the number of extensions is exactly $\inj(H[S],K_{x,y})$.

The graph $H[S]$ is claw-free because it is induced in $H$. 
By Lemma~\ref{claw:lem:two-part}, we get
$$\inj(H[S],K_{y,x})\le \inj(H[S],K_{y+1,x-1}).$$
Since $K_{y,x}$ is isomorphic to $K_{x,y}$, this shows that replacing the two part sizes $(x,y)$ by $(x-1,y+1)$ does not decrease the number of extensions of any fixed pair $(S,\psi)$. Summing over all $S$ and all such $\psi$, we conclude that the total number of injective homomorphisms does not decrease.

Repeatedly apply this operation whenever two part sizes differ by at least two. The process terminates because the sum of the squares of the part sizes strictly decreases at every step. At termination all part sizes differ by at most one, so the resulting graph is $T_k(n)$. This proves
$$\inj(H,K_{n_1,\ldots,n_k})\le \inj(H,T_k(n)).$$
Finally, every unlabeled copy of $H$ gives exactly $|\mathrm{Aut}(H)|$ injective homomorphisms, and therefore
$$\inj(H,G)=|\mathrm{Aut}(H)|\N(H,G)$$
for every graph $G$. Dividing by $|\mathrm{Aut}(H)|$ proves the equivalent assertion for copies.
\end{proof}

\section{Concluding remarks}\label{sec:remarks}

The proof of Theorems~\ref{thm:cubic} and~\ref{thm:alln} also shows what must be improved to obtain a quadratic threshold. 
After the degree sensitive symmetrization, the minimum degree deficit has order $v(H)/r$. 
The sharpened multipartite lemma requires the density parameter to have order at most $v(H)^{-2}$. 
Combining these two bounds yields an upper bound of order $v(H)^3$ for the threshold $r_0(H)$.
A quadratic bound would require either a rebalancing error of order $O(\beta v(H))$ or a minimum-degree deficit of order $O(1/r)$, 
neither of which follows from the present estimates.

%Lemma~\ref{cub:lem:transfer} yields the smaller coefficient $2v(H)^2+2v(H)-4+v(H)^3/e(H)$ in place of $4v(H)^2+2v(H)-4$. 
%Thus denser graphs $H$ give better constants within the same argument. 
%We have stated the uniform bound directly in terms of $v(H)$ in order to keep the main theorem transparent.

By the works of \cite{Gerbner2023,GerbnerKarim,HeiHouLiu}, Theorem~\ref{thm:claw-stable} immediately implies the following:
(1). Let $M$ be a matching in $K_m$. If $\chi(K_m-M)\le k$, then $K_m-M$ is $K_{k+1}$-Tur\'an-good.
(2). If every nontrivial connected component of a graph $H$ is a path, a clique, or a copy of $C_4$, and $\chi(H)\le k$, then $H$ is $K_{k+1}$-Tur\'an-good; arbitrary isolated vertices are allowed.

We also propose the following question: Characterize the graphs $H$ such that if $H$ is $K_r$-Tur\'an-good, then $H$ is also $K_{r+1}$-Tur\'an-good.

\section*{Acknowledgements}
\noindent 
This article was submitted in July 2026. The same problem was considered independently by Wang, Kang, and Zhao \cite{Wang2026}.
During the preparation of this work, the author used OpenAI's ChatGPT (GPT-5.6 Pro) to assist with language editing, organization, and preliminary checks of algebraic and logical consistency.
This research is supported by National Key R\&D Program of China under grant number 2024YFA1013900, NSFC under grant number 12471327, 
and the China Postdoctoral Science Foundation under Grant Number 2026M793375.

\appendix

\section{Two further counterexample families}\label{app:further-counterexamples}

We now define and prove the two further families mentioned in the introduction. Fix $r\ge3$ and let $A_1,\ldots,A_{r-1}$ be disjoint sets. For $s\ge1$ and $1\le t\le r-2$, let $G_{r,s,t}$ be the complete $(r-1)$-partite graph with $t$ parts of size $s+1$ and $r-1-t$ parts of size $s$. Thus $G_{r,s,t}=T_{r-1}((r-1)s+t)$. For $s\ge3$ and $2\le d\le s-1$, let $J_{r,s,d}$ be obtained from the equal-part complete $(r-1)$-partite graph by choosing $u\in A_1$ and distinct vertices $v_1,\ldots,v_d\in A_2$ and deleting $uv_1,\ldots,uv_d$.

\begin{theorem}\label{app:thm:further-counterexamples}
Let $r\ge3$ and $s\ge12\lceil\log_2r\rceil$. Then every $G_{r,s,t}$ with $1\le t\le r-2$ and every $J_{r,s,d}$ with $2\le d\le s-1$ is $K_r$-Tur\'an-good but is not $K_{r+1}$-Tur\'an-good.
\end{theorem}

The graphs $G_{r,s,t}$ are complete multipartite with unequal parts and the graph $J_{r,s,d}$ is not complete multipartite. 
Moreover, $J_{r,s,d}$ is not isomorphic to any $H_{r,s',\ell}$. 
Equality of their orders would give $s'=s$, but $u$ has degree $(r-2)s-d$ in $J_{r,s,d}$, while every vertex of $H_{r,s,\ell}$ has degree at least $(r-2)s-1$.

\subsection{The \texorpdfstring{$K_r$}{K-r}-Tur\'an-good property}

The following two proofs use the theorem of Gy\H{o}ri, Pach, and Simonovits \cite{Gyori1991} stated as Theorem~\ref{ce:thm:Gyori1991}.

\begin{lemma}\label{ce:lem:G-good}
Let $r\geq3$, $s\geq1$, and $1\leq t\leq r-2$. Then $G_{r,s,t}$ is $K_r$-Tur\'an-good.
\end{lemma}

\begin{proof}
Set $q=r-1$. The graph $G_{r,s,t}$ is a complete $q$-partite graph on $qs+t$ vertices, so $\lfloor(qs+t)/q\rfloor=s$. Choose $s$ vertices from every part and label the chosen vertices in each part by $1,\ldots,s$. For every $1\leq i\leq s$, the vertices labeled $i$, one from each part, form a copy of $K_q$. These $s$ copies are vertex-disjoint.

Every copy of $K_q$ contains exactly one vertex from each part. Given two such copies, change their vertices one part at a time. Every intermediate set containing one vertex from each part is again a copy of $K_q$, and two consecutive copies have exactly $q-1$ common vertices. Moreover, every vertex belongs to a copy of $K_q$ containing one vertex from each part. Thus the $q$-skeleton is connected. Theorem~\ref{ce:thm:Gyori1991} now shows that $G_{r,s,t}$ is $K_r$-Tur\'an-good.
\end{proof}

\begin{lemma}\label{ce:lem:J-good}
Let $r\geq3$, $s\geq3$, and $1\leq d\leq s-1$. Then $J_{r,s,d}$ is $K_r$-Tur\'an-good.
\end{lemma}

\begin{proof}
Set $q=r-1$, let the parts be $A_1,\ldots,A_q$, and let the deleted edges be $uv_1,\ldots,uv_d$, where $u\in A_1$ and $v_1,\ldots,v_d\in A_2$. Choose $y\in A_2\setminus\{v_1,\ldots,v_d\}$. There is a bijection from $A_1$ to $A_2$ that maps $u$ to $y$. Choose arbitrary bijections from $A_1$ to each of $A_3,\ldots,A_q$. For each $x\in A_1$, take $x$ together with its images under these bijections. The resulting $s$ sets are vertex-disjoint copies of $K_q$: the only deleted edges are incident with $u$, and the vertex of $A_2$ paired with $u$ is $y$. Since the graph has $qs$ vertices, this gives $\lfloor qs/q\rfloor=s$ vertex-disjoint copies of $K_q$.

Choose $x^*\in A_1\setminus\{u\}$, keep the vertex $y\in A_2\setminus\{v_1,\ldots,v_d\}$ fixed, and choose one fixed vertex in every other part. These vertices form a copy $Q^*$ of $K_q$. Let $Q$ be any copy of $K_q$. First replace its vertex in $A_1$ by $x^*$; this is valid because $x^*$ is adjacent to every vertex of $A_2$. Next replace its vertex in $A_2$ by $y$, and then replace the vertices in the remaining parts one at a time by those of $Q^*$. Every intermediate set is a copy of $K_q$, so every copy of $K_q$ is joined to $Q^*$ in the $q$-skeleton.

Every vertex belongs to a copy of $K_q$. The vertex $u$ can be used with $y$; every vertex of $A_2$ can be used with $x^*$; and a vertex in any other part can be used together with $x^*$, $y$, and arbitrary vertices from the remaining parts. Therefore the $q$-skeleton is connected. Theorem~\ref{ce:thm:Gyori1991} implies that $J_{r,s,d}$ is $K_r$-Tur\'an-good.
\end{proof}

\begin{lemma}\label{ce:lem:G-color}
Let $q\geq2$, $s\geq1$, and $1\leq t\leq q-1$. For $G=G_{q+1,s,t}$,
$$P_G(q)=q!,$$
$$P_G(q+1)=q!\left[(q+1)+\frac{q+1}{2}\left(t(2^{s+1}-2)+(q-t)(2^s-2)\right)\right].$$
\end{lemma}

\begin{proof}
Different parts of a complete multipartite graph must use disjoint sets of colors. In a proper $q$-coloring of $G$, the $q$ nonempty parts therefore use exactly one color each, giving $P_G(q)=q!$.

For a proper $(q+1)$-coloring, first suppose that exactly $q$ colors are used. Choose the unused color and assign the other colors bijectively to the parts; this gives $(q+1)q!$ colorings. If all $q+1$ colors are used, exactly one part uses two colors and all other parts use one. A part of size $a$ contributes $\binom{q+1}{2}(q-1)!(2^a-2)$ colorings. There are $t$ parts of size $s+1$ and $q-t$ parts of size $s$. Summing these contributions and using $\binom{q+1}{2}(q-1)!=q!(q+1)/2$ gives the stated formula.
\end{proof}

\begin{lemma}\label{ce:lem:J-color}
Let $q\geq2$, $s\geq3$, and $1\leq d\leq s-1$. For $J=J_{q+1,s,d}$,
$$P_J(q)=q!\quad \text{and} \quad P_J(q+1)=q!\left[(q+1)+\binom{q+1}{2}(2^s-2)+(q+1)(2^d-1)\right].$$
\end{lemma}

\begin{proof}
Let the deleted edges be $uv_1,\ldots,uv_d$, where $u\in A_1$, and set $D=\{v_1,\ldots,v_d\}\subseteq A_2$. The sets $A_1\setminus\{u\}$, $A_2\setminus D$, and $A_3,\ldots,A_q$ are all nonempty and induce a complete $q$-partite graph.

Consider a proper $q$-coloring. If $u$ and some $v_i$ had the same color, then this color could not be used on any vertex of the complete $q$-partite graph just described: vertices of $A_1\setminus\{u\}$ are adjacent to $v_i$, vertices of $A_2\setminus D$ are adjacent to $u$, and vertices in the other parts are adjacent to both. That complete $q$-partite graph contains $K_q$ and would have to use at most $q-1$ colors, a contradiction. Hence no deleted pair is monochromatic. The coloring is then a proper $q$-coloring of the complete $q$-partite graph with parts $A_1,\ldots,A_q$, and therefore $P_J(q)=q!$.

For $(q+1)$ colors, first count colorings in which no deleted pair is monochromatic. These are precisely the proper colorings of the complete $q$-partite graph with equal parts, and their number is $q![(q+1)+\binom{q+1}{2}(2^s-2)]$ by the same argument as above.

Now suppose that at least one deleted pair is monochromatic. Let $c$ be the color of $u$ and let $S=\{v_i\in D:v_i\text{ has color }c\}$. Then $S$ is a nonempty subset of $D$. The color $c$ occurs exactly on $\{u\}\cup S$: it cannot occur in $A_1\setminus\{u\}$ because these vertices are adjacent to every vertex of $S$, it cannot occur in $A_2\setminus D$ because these vertices are adjacent to $u$, and vertices in the other parts are adjacent to $u$. Conversely, after choosing $c$ and a nonempty subset $S\subseteq D$, the remaining graph is the complete $q$-partite graph with nonempty parts $A_1\setminus\{u\}$, $A_2\setminus S$, $A_3,\ldots,A_q$. It must be colored with the remaining $q$ colors, one color on each part, in $q!$ ways. Thus this case contributes $(q+1)(2^d-1)q!$ colorings. Adding both cases proves the formula.
\end{proof}

\subsection{Proof of Theorem~\ref{app:thm:further-counterexamples}}

\begin{proof}
Setting $q=r-1$. 
Lemmas~\ref{ce:lem:G-good} and~\ref{ce:lem:J-good} show that all graphs in the theorem are $K_r=K_{q+1}$-Tur\'an-good. 
We use Corollary~\ref{ce:cor:criterion} and Lemma~\ref{ce:lem:growth} to prove that they are not $K_{r+1}$-Tur\'an-good.

Let $G=G_{r,s,t}$. It has $qs+t$ vertices. By Lemma~\ref{ce:lem:G-color}, it is enough to prove
$$\left(\frac r{r-1}\right)^{(r-1)s+t}>r+\frac r2\left(t(2^{s+1}-2)+(r-1-t)(2^s-2)\right).$$
The left-hand side is larger than the left-hand side in Lemma~\ref{ce:lem:growth}. Since $1\le t\le r-2$, the right-hand side is less than $r+\frac r2(2r-3)2^s<r^22^s$. Hence the displayed inequality holds, and Corollary~\ref{ce:cor:criterion} shows that $G$ is not $K_{r+1}$-Tur\'an-good.

Now let $J=J_{r,s,d}$. It has $qs$ vertices. Lemma~\ref{ce:lem:J-color} reduces the required inequality to
$$\left(\frac r{r-1}\right)^{(r-1)s}>r+\binom r2(2^s-2)+r(2^d-1).$$
Because $d\le s-1$, the right-hand side is less than
$$r+\frac{r(r-1)}2\,2^s+r2^{s-1}=r+\frac{r^2}{2}\,2^s<r^22^s.$$
Lemma~\ref{ce:lem:growth} and Corollary~\ref{ce:cor:criterion} imply that $J$ is not $K_{r+1}$-Tur\'an-good. 

In both cases, we are done.
\end{proof}
\end{document}